\documentclass{svjour3}
\usepackage[OT2,OT1]{fontenc}
\newcommand\cyr
{
\renewcommand\rmdefault{wncyr}
\renewcommand\sfdefault{wncyss}
\renewcommand\encodingdefault{OT2}
\normalfont
\selectfont
}
\DeclareTextFontCommand{\textcyr}{\cyr}
\def\cprime{\char"7E }

\usepackage{graphics}
\usepackage{tikz}
\usepackage{pgf}
\usetikzlibrary{arrows,automata}
  \usepackage{color}
  \definecolor{purple}{rgb}{0.5,0,1}
  \definecolor{orange}{cmyk}{0,0.7,1,0}
  \definecolor{midgrey}{gray}{0.5}

  \usepackage{amssymb}
    \usepackage{amsthm}
  \usepackage{amsmath}
  \usepackage{epstopdf}

\newcommand{\bd}{\begin{displaymath}}
\newcommand{\ed}{\end{displaymath}}

\newcommand {\bdm}{\begin{displaymath}}
\newcommand {\edm}{\end{displaymath}}

\newcommand{\ti}{\widetilde}

\newcommand {\bad}{\begin{aligned}}
\newcommand {\ead}{\end{aligned}}
\newcommand {\ben}{\begin{equation}}
\newcommand {\een}{\end{equation}}
\newcommand {\bay}{\begin{array}}
\newcommand {\eay}{\end{array}}
  \makeatletter
  \let\c@equation\undefined
  \let\c@section\undefined
  \let\c@subsection\undefined
  \let\c@zad\undefined
  \makeatother
  \newcounter{section}
  \newcounter{equation}
  
  \newcounter{subsection}[section]

\newcommand{\beq}{\begin{equation}}
\newcommand{\eeq}{\end{equation}}

\newcommand{\papap}{\papa \end{proof}}

\newcommand{\mr}[1]{\mathring{#1}}

\newfont{\smoldita}{cmmib8}
\newfont{\boldita}{cmmib10}
\newfont{\bboldita}{cmmib10}

\newcommand{\nn}{\nonumber}
\newcommand{\e}{\epsilon}

\newcommand{\cl}[2]{\int\limits_{#1}^{#2}}

\newcommand{\la}{\lambda}

\newcommand{\mb}[1]{\boldsymbol{#1}}
\newcommand{\mbb}[1]{\mathbb{#1}}
\newcommand{\mc}[1]{\mathcal{#1}}
\newcommand{\be}{\begin{eqnarray}}
\newcommand{\ee}{\end{eqnarray}}
\newcommand{\ups}{\bar{\mb\upsilon}}

\begin{document}
\title{A note on the Tikhonov theorem on an infinite interval\thanks{Research partially supported by National Science Centre, Poland, grant 2017/25/B/ST1/00051 and the National Research Foundation of South Africa, SARChI grant 8277} }
\author{Jacek Banasiak }

 \institute{Jacek Banasiak \at Department of Mathematics and Applied Mathematics, University of Pretoria, South Africa\\ Institute of Mathematics, \L\'od\'z University of Technology, Poland, \\ Scientific Research Laboratory of Applied Semigroup Research, South Ural State University, Russia\\e-mail:jacek.banasiak@up.ac.za}

\date{ }

\maketitle
\begin{abstract}
In this note we provide a new proof of the Tikhonov theorem for the infinite time interval and discuss some of its applications.
\end{abstract}
%\tableofcontents
\section{Introduction}
 Modern modelling dynamical processes with ordinary differential equations usually leads to very large and complex systems with the coefficients that often widely differ in magnitude. These features make any robust analysis of them close to impossible. In particular, the presence of very small and very large coefficients creates a stiffness in the system that renders standard numerical methods unreliable. At the same time, the presence of such coefficients indicates that the process is driven by mechanisms acting on very different time scales. Then one can hope that there is a dominant time scale; that is, the time scale at which the system, obtained by an appropriate aggregation of much faster and/or much slower processes, will have the same main dynamical features as the original one.

The presence of different time scales in a system is revealed if the nondimensionalization with respect to the chosen reference time unit produces coefficients that are significantly larger (or smaller) than the others. In this paper we will be dealing with systems that can be written in the so-called canonical, or Tikhonov, form
\begin{align}
\mb u_{\e,t}&={\mb f}(\mb u_\e,\mb v_\e,t,\e), \quad \mb u_\e(0) =\mr {\mb u},\nn\\
\e \mb v_{\e,t} &={\mb g}(\mb u_\e,\mb v_\e,t,\e),\quad \mb v_\e(0)= \mr{\mb  v},
\label{CE10}
\end{align}
where $\mbox{}_{,t}$ denotes the time derivative, $\mb f$ and $\mb g$ are sufficiently smooth functions acting from an open subset of $\mbb R^n\times\mbb R^m$  into, respectively, $\mbb R^n$ and $\mbb R^m,$  and $\e$ is a small parameter. As we shall see below,  many more complex systems can be brought to such a form, see e.g. \cite{BanLach} for a systematic approach to a class of such systems. The interpretation of \eqref{CE10} is that the processes described by $\mb g$ happen much faster than those described by $\mb f$ and thus, if we are interested in larger times, it is plausible to assume that the former reach an equilibrium before any significant change occurs in the latter. Hence, for small $\e$, the solution to \eqref{CE10} should be close to the pair consisting of the solution $\bar{\mb v}(t) = \mb{\phi}(\mb u, t)$ to the algebraic equation
\begin{equation}
{\mb g}(\mb u,\mb v,t,0) =0,
\label{qss10}
\end{equation}
called the \textit{quasi steady state}, and the solution $\bar{\mb u}(t)$ of the \textit{reduced equation}
\begin{equation}
\mb u_{,t}={\mb f}(\mb u,\mb\phi(\mb u,t),t,0),\quad \mb u(0) =\mr {\mb u}.
\label{bulk1}
\end{equation}
The validity of such an approximation is the subject of the Tikhonov theorem, see e.g. \cite{TVS,VasBut}, that was also proven by other methods, such as the Center Manifold Theory, see e.g. \cite{Carr,Fen}. The two main assumptions of the Tikhonov theorem are that: a)    \eqref{qss10} admits a single isolated solution and, noting that for any fixed $(\mb u,t)$, $\mb{\phi}(\mb u, t)$ is an equilibrium of the initial layer equation
\begin{equation}
\ti{\mb v}_{\e,\tau} ={\mb g}(\mb u,\ti{\mb v},t,0),
\label{il1}
\end{equation}
where $\tau = t/\e$ but $(\mb u,t)$ are treated as parameters, and b) $\mb{\phi}(\mb u, t)$ is uniformly asymptotically stable in $(\mb u,t).$ The latter assumption is a mathematical expression of the fact that  the quasi steady state is indeed practically reached in the fast time $t/\e$; that is, almost immediately in terms of the slow time $t$.

A problem with the Tikhonov theorem is that the adopted assumptions only ensure that the convergence is valid on finite intervals of time $t$ and thus the approximation \eqref{qss10}, \eqref{bulk1} is useless if one wants to investigate the long term dynamics of \eqref{CE10}. In other words, within the framework of the Tikhonov theorem, one cannot substitute  the quasi steady state into \eqref{CE10} and draw any valid conclusions about the long term dynamics of \eqref{CE10} from the resulting reduced equation. We shall present an example of such a situation in Section \ref{sec1}.

This problem was first addressed in \cite{Hop} where, assuming additionally that the relevant equilibrium of \eqref{bulk1} is uniformly asymptotically stable,  the author used the reverse Lyapunov theorem of \cite{Mass} to construct appropriate Lyapunov functions to push the estimates of the original Tikhonov's proof to infinity. Recently the problem was again picked in \cite{MarCz}, where the authors used the ideas of \cite{Carr} to localize the equations along the quasi steady state and then proved the uniform in time estimates by directly using differential inequalities. The tricky part of this method is that the localization must preserve the properties of $\mb g$ that allow for solving \eqref{qss10} and keep the stability of the localized version of equation  \eqref{bulk1}. Moreover, the localization must be extended close to $t=0$ to a funnel-like domain  to encompass initial conditions $\mr{\mb y}$ that may be far away from the quasi steady state. These, together with the specific form of localization, forced the authors of \cite{MarCz} to consider a restricted form of $\mb g$ which is one-dimensional with a dominant constant coefficient linear part.

The main aim of this paper is to address the restrictions mentioned above. We follow the ideas of \cite{MarCz} but use a different form of the localization that preserves the linear part of $\mb f$ and $\mb g.$ Moreover, to avoid extending the localization to the  funnel-like domain (which requires an additional assumption), we use the  estimates of the original proof of the Tikhonov theorem close to $t=0$ and then employ the differential inequalities only for small initial conditions.

In the present paper we only prove the Tikhonov type result without addressing the order of the convergence, as was done in \cite{MarCz}. The higher order estimates, that can be done using the same approach,  are subject of the forthcoming paper. We also note that our assumptions, while being in line with that of \cite{MarCz}, are stronger then in the original Tikhonov theorem but they are satisfied in most applications and make the proofs more straightforward.

\noindent
\textbf{Acknowledgements.} The author is grateful to Anna Marciniak-Czochra and Thomas Stiehl for bringing his attention to the applications of the localization method in the singular perturbation theory.

\section{Example}
\label{sec1}

In this section, following  \cite[Section 4.3]{BanLach} and \cite{Aug}, we present an example showing that the Tikhonov approximation, being valid on finite time intervals, may not provide any reliable information on the long term dynamics of the original system.
\begin{example}
 We assume that we have the populations of prey and predators, where  the prey can move between two locations, say, grazing grounds and some refuge, while the predators only hunt in the grazing area. The interactions between the predators and the prey are modelled by the mass-action law. We denote by $(n_1,n_2)$  and $p$, respectively,  the prey and the predator populations, and assume that the migrations are fast if compared to the vital
processes. Then
\begin{eqnarray}
  n_{1,t} &=& n_{1}(r_{1}-ap)+ \frac{1}{\e}(m_{2}n_{2}-m_{1}n_{1}),\nn  \\
  n_{2,t} &=& n_{2}r_{2} +\frac{1}{\e}(m_{1}n_{1}-m_{2}n_{2}),  \nn\\
  p_{,t} &=& p(bn_{1}-d),
\label{eq1.1}
\end{eqnarray}
where, for $i=1,2,$ $n_{i}$ denotes the prey density in patch $i$,
$m_{i}$ denotes the migration rate from patch $i,$
 $r_{i}$ is the prey population growth
rate in patch $i$, $d$ is the predator death rate, $a$
is the predation rate in patch $1$ and $b$ is the biomass conversion rate.

We note that \eqref{eq1.1} is not in the typical Tikhonov form as letting $\e=0$ in the first two equations yields two identical equation and the assumptions of the Tikhonov theorem are not satisfied. However, adding the first two equation  and introducing the total prey population $n:=n_1+n_2,$
we obtain
\begin{eqnarray}
{n}_{,t} &=&  n(r_1-ap) + n_2(r_2-r_1-ap),\nn\\
\e n_{2,t}&=& \e n_2r_2+  (m_1 n -n_2(m_1+m_2)),\nn\\
{p}_{,t} &=& p(bn -bn_2 -d).
\label{LVsp}
\end{eqnarray}
Denoting $
M_i = \frac{m_i}{m_1+m_2},\quad i=1,2,$
we get the quasi steady state $\bar n_2 = M_1 \bar n$ and the reduced system
\begin{align}
{\bar n}_{,t} &=  \bar n(\bar r -aM_1\bar p),\nn\\
\bar p_{,t}&= (bM_2 \bar n - d),\label{LV1}
\end{align}
which we recognize as the Lotka-Volterra model with the aggregated birth rate for the prey $\bar r= M_2r_1 +M_1r_2$ and similarly adjusted predation and biomass conversion rates.  We see that the assumptions of the Tikhonov theorem are satisfied as
the quasi steady state $ \bar n_2 = M_1 \bar n$ is a uniformly asymptotically stable equilibrium of the initial layer equation
$$
 \ti n_{2,\tau}= m_1 n -\ti n_2(m_1+m_2).
$$
Thus the solution $(\bar n, \bar p)$ (augmented by the initial layer) approximates the solution $(n_1+n_2, p)$ of (\ref{eq1.1}) on finite time intervals. On the other hand, the equilibria of \eqref{eq1.1} are
 $(0,0,0)$  and, for small $\e$,
 $$
 (n^*_1, n^*_2, p^*) = \left(\frac{d}{b}, \frac{m_1 d}{b(m_2-\e r_2)},  \frac{r_1}{a} + \frac{m_1r_2}{a(m_2-\e r_2)}\right)
 $$
The Jacobi matrix of (\ref{eq1.1}), evaluated at $(n_1^*,n_2^*, p^*)$, gives
$$
\mc J^* = \e^{-1}\left(\begin{array}{ccc}-\frac{m_1m_2}{(m_2-\e r_2)}&m_2&-\frac{\e ad}{b}\\
m_1&\e r_2-m_2&0\\
\e b\left(\frac{r_1}{a} + \frac{m_1r_2}{a(m_2-\e r_2)}\right)&0&0\end{array}\right).
$$
Denoting $\alpha = {m_1m_2}/{(m_2-\e r_2)}$, $\beta = ad/b,$  $\gamma = bp^*$ and using $\alpha m_2-m_1m_2 =  \e\alpha r_2$, we get the characteristic equation
$$
\la^3 +\la^2(\alpha +m_2-\e r_2)+\la \e^2\beta\gamma +\e^2\beta\gamma (m_2-\e r_2)=0.
$$
For small $\e>0$ all coefficients are positive and hence e.g. the Hurwitz criterion
ensures that real parts of all eigenvalues of $\mc J^*$ are negative. Thus the positive equilibrium of the system (\ref{eq1.1}) is asymptotically stable.
\index{asymptotically stable solution} This is in contrast to (\ref{LV1}), for which the positive equilibrium is a centre.

On Fig. \ref{fig:8}, where we use  $m_1=2, m_2=1, r_1=1, r_2=2, a=1, d=1, b=0.9$ and $\e = 0.05$, we see that the approximation is initially good but loses accuracy for larger times.
\begin{figure}
\begin{center}
  \includegraphics[width=5cm,height=7cm,angle=270,bb=50 50 554 770,clip=true,trim=2pt 0pt 0pt 0pt]{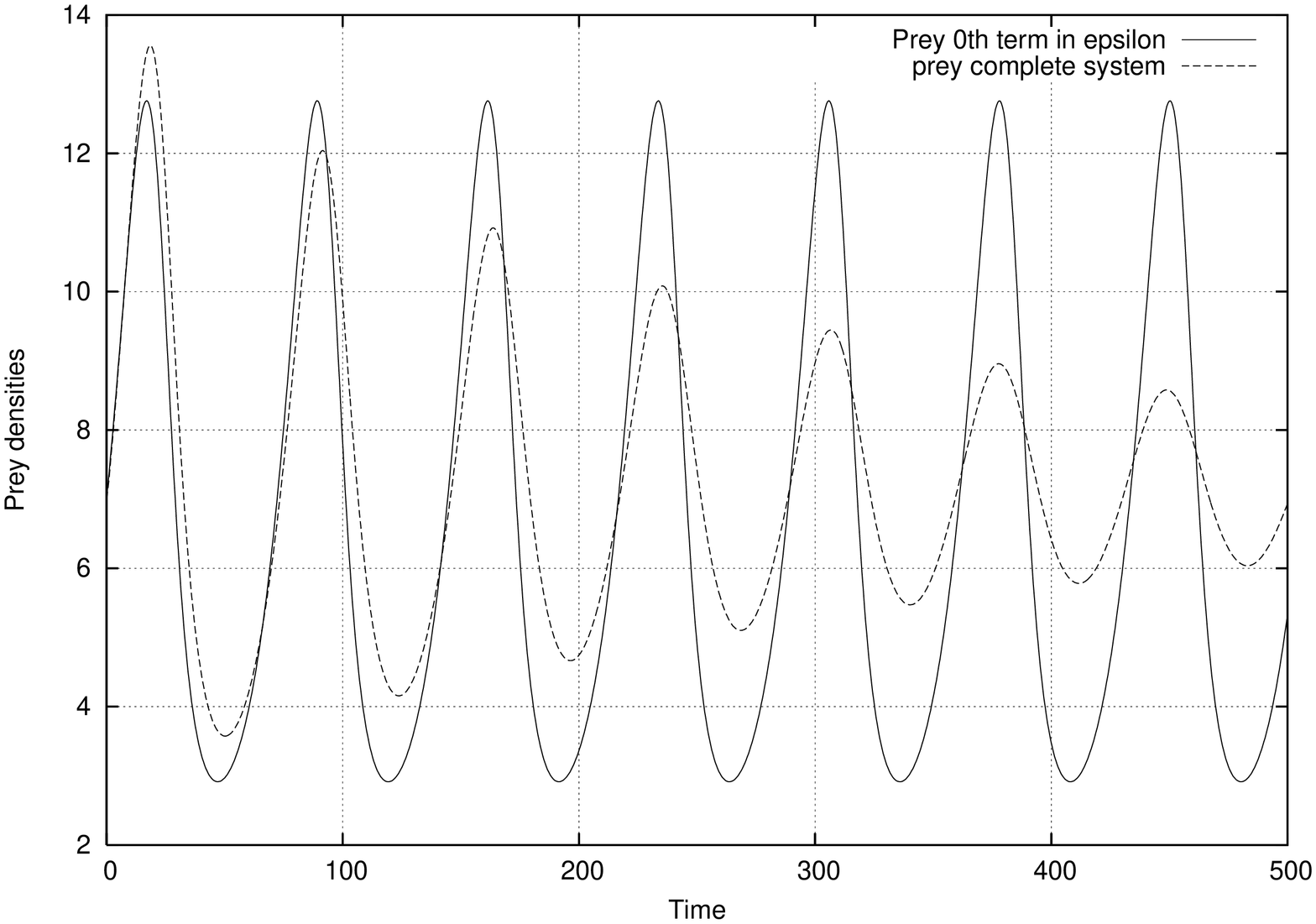}\\
  \includegraphics[width=5cm,height=7cm,angle=270,bb=50 50 554 770,clip=true,trim=2pt 0pt 0pt 0pt]{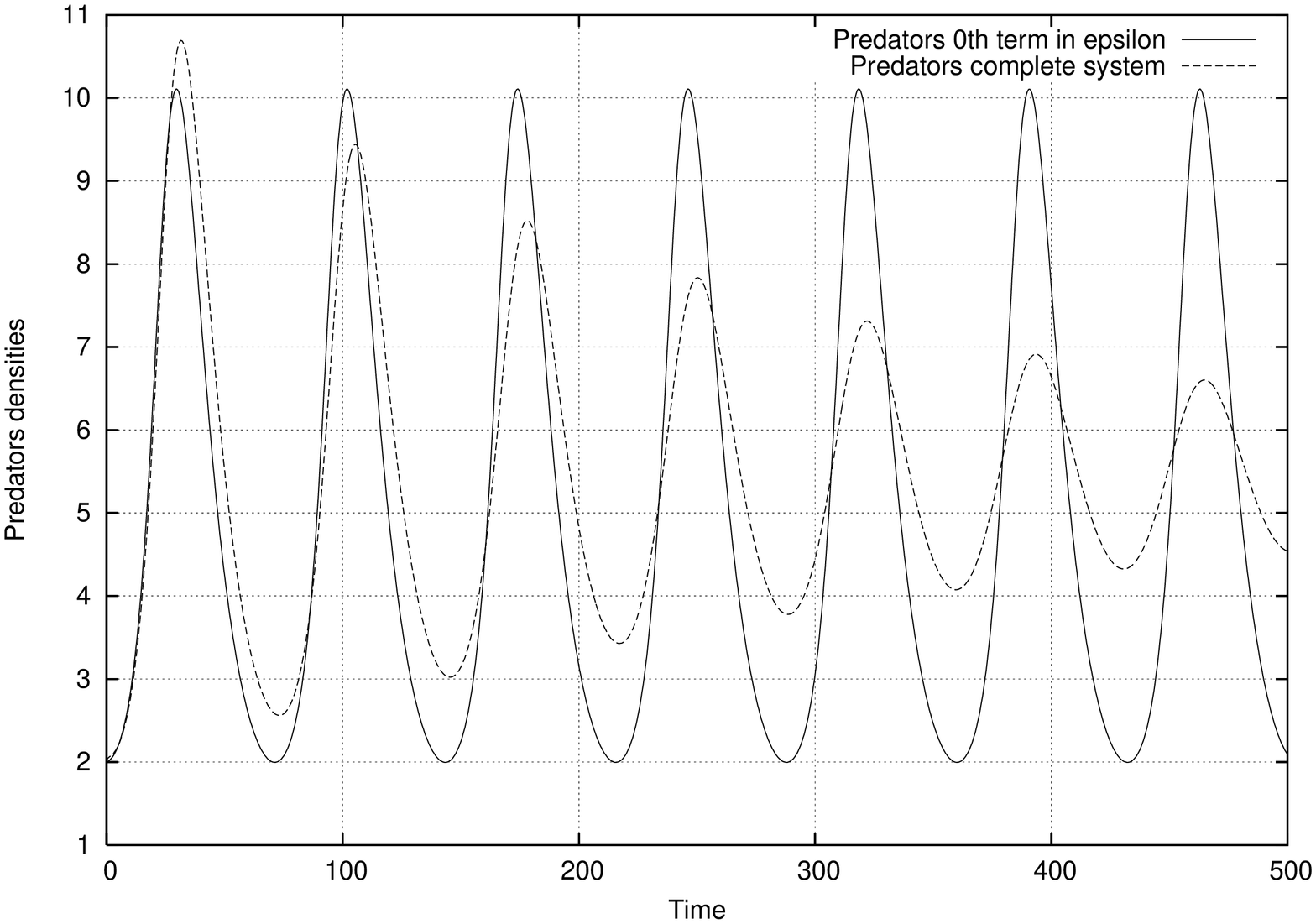}
 \end{center}
  \caption{Comparisons of the  prey $n$ (top) and predator $p$  (bottom) populations given by (\ref{eq1.1}) (dashed line) with the approximating  populations given by (\ref{LV1}) (solid line).}
\label{fig:8}
\end{figure}

\end{example}
\section{Notation and assumptions}

As mentioned in the introduction, we consider an $n\times m$ dimensional system
\begin{align}
\mb u_{\e,t}&={\mb f}(\mb u_\e,\mb v_\e,t,\e), \quad \mb u_\e(0) =\mr {\mb u},\nn\\
\e \mb v_{\e,t} &={\mb g}(\mb u_\e,\mb v_\e,t,\e),\quad \mb v_\e(0)= \mr{\mb  v}.
\label{CE1}
\end{align}
To be consistent with writing systems of equations in the column form, we adopt the convention that any vector $\mb x$ is a column vector and thus for a function $t\mapsto \mb x(t)$, $\mb x_{,t}$ is a column vector. Then, for a scalar function $\mb x \mapsto h(\mb x),$  $h_{,\mb x}$ is the row vector of the first derivatives of $h$ and,  for an $\mbb R^m$ valued  vector function  $\mbb R^n\ni \mb x \mapsto \mb h(\mb x) = (h_1(\mb x),\ldots, h_m(\mb x)),$  $\mb h_{,\mb x}$  the matrix having $h_{i,\mb x}, i=1,\ldots, m,$ as its rows, or
\begin{equation}
\mb h_{,\mb x} = (h_{i,x_j})_{1\leq i\leq m, 1\leq j\leq n}.
\label{conv}
\end{equation}
The estimates for the  nonlinear problems depend to large extent on the estimates for their time dependent linearizations. For this we recall some relevant results. Consider an $r\times r$ system  on $R_+$,
\begin{align}
\mb x_{,t}(t) &= A(t)\mb x(t),\quad \mb x(t_0) =\mr{\mb x},
\label{LS1}
\end{align}
where $A$ is a continuous matrix function, and let $\mbb R_+\ni t\to Y_A(t)$ be the fundamental matrix for \eqref{LS1}. We say that $A$ satisfies the exponential dichotomy property if
there are positive constants $K,\alpha$ such that
\begin{equation}
\|Y_A(t)Y^{-1}_A(s)\|\leq Ke^{-\alpha(t-s)}, \quad t_0\leq s\leq t<\infty.
\label{expdich1}
\end{equation}
We note that this is a simplified case of the exponential dichotomy discussed in e.g. \cite[Section 2.1]{Lin}. It is equivalent to the uniform exponential stability of \eqref{LS1}, see \cite[Theorem III.1]{Cop}. By \cite[Corollary (I) to Theorem III.9]{Cop} and \cite[Theorem 2.6]{Lin}, this property is stable under small, or vanishing at infinity, continuous perturbations.

Let us introduce assumptions on \eqref{CE1}.

\noindent
(A1) We assume that $\mb f:\mbb R^{n+m}\times \mbb R_+\times I_e\mapsto \mbb R^n$ and $\mb g:\mbb R^{n+m}\times \mbb R_+\times I_{e}\mapsto \mbb R^m,$ where $I = [0,e], e>0$, are $C^3$ functions with respect to $\mb u$ and $\mb v$ and $C^1$ with respect to $t$ and $\e,$ that are bounded together with all existing derivatives on $[0,\infty)$ uniformly for $\mb u,\mb v$ in bounded subsets of $\mbb R^{n+m}$ and $\e \in I_a$.

As in the classical Tikhonov theorem, we assume that \\
(A2)
\begin{equation}
0= {\mb g}({\mb u},{\mb v},t,0)
\label{degeq1}
\end{equation}
admits an isolated solution ${\mb v}(t) = {\mb \phi}({\mb u},t)$ for any $(\mb u, t)\in \mbb R^n\times [0,\infty)$ and we denote by $\bar {\mb u}$ the solution to
\begin{equation}
{\mb u}_{,t}={\mb f}({\mb u},{\mb \phi}({\mb u},t),t,0), \quad {\mb u}(0) = \mr {\mb u}.
\label{deqeq2}
\end{equation}
We assume that
\begin{equation}
t\mapsto \bar{\mb \upsilon}(t):=(\bar{\mb u}(t), \mb{\phi}(\bar{\mb u}(t),t))
\label{up1}
\end{equation}
 is a bounded differentiable function on $[0,\infty)$.

For any matrix $A$ we denote by $\sigma(A)$ the spectrum of $A$ and by $s(A) := \max\{\Re \la;\; \la\in \sigma(A)\}$ the spectral bound of $A$.

\noindent
(A3) For the  matrix $\mb g_{,\mb v}$ evaluated along $\bar{\mb \upsilon}(t)$ we assume that
\begin{equation}
\sup\limits_{t\in [0,\infty)} s({\mb g}_{,\mb v}(\ups(t),0))=: -\kappa'<0.
\label{mug}
\end{equation}
\begin{remark} Thanks to the assumption that ${\mb g}$ is a $C^3$ function with respect to $(\mb u, \mb v)$ and to the continuity of eigenvalues, see e.g. \cite[Section 3.1]{Ort}, there is $\delta>0,$ $\kappa>0$ and $\e_0>0$ such that \begin{equation}
\sup\limits_{(\mb u,\mb v)\in T_{\delta}, \e \in I_{\e_0} }s({\mb g}_{,\mb v}(\mb u, \mb v ,t,\e))=: -\kappa<0,
\label{tubeest}
\end{equation}
where
\begin{equation}
T_{\delta} = \bigcup\limits_{0\leq t<\infty}E_{\delta}(\ups(t)),
\label{tubedef}
\end{equation}
and
\begin{align*}
E_{\delta} (\ups( t)) &:= E_{\delta} (\bar {\mb u}(t), t)\times  E_{\delta} (\phi(\bar{\mb u}(t), t):= \{(\mb u,\mb v);\;\|\mb u-\bar {\mb u}(t)\|\leq \delta, \|\mb v-\mb \phi(\bar {\mb u}(t),t)\|\leq \delta\}.
\end{align*}
\end{remark}
Next, consider the auxiliary equation \begin{align}
\hat {\mb v}_{0,\tau}(\tau) &= {\mb g}(\mathring {\mb u},  \hat {\mb v}_0(\tau),0,0),\quad \hat {\mb v}_0(0)= \mathring {\mb v},\label{hav0}
\end{align}
(A4) We assume that $\mr{\mb v}$ belongs to the basin of attraction of the root $\mb \phi(\mr{\mb u},0)$; that is,
the solution $\hat{\mb v}_0$ of \eqref{hav0} satisfies
\begin{equation}
\lim\limits_{\tau\to \infty} \hat{\mb v}_0(\tau)= \mb \phi(\mr{\mb u},0).
\label{A5}
\end{equation}
\begin{remark}\label{rem2}
As in the classical Tikhonov theorem, it follows then that, under \eqref{mug}, there is a constant $C>0$ such that
\begin{equation}
\|\hat{\mb v}_0(\tau)- \mb \phi(\mr{\mb u},0)\| \leq Ce^{-\kappa \tau}, \quad \tau\geq 0.
\label{ilest1}
\end{equation}
Indeed, defining
\begin{equation}
\ti{\mb v}_0(\tau)=  \hat{\mb v}_0(\tau) -\mb\phi(\mb{\mr u}, 0),
 \label{hatv}
 \end{equation}
we have
\begin{align}
\ti {\mb v}_{0,\tau}(\tau) &= {\mb g}(\mathring {\mb u}, {\mb \phi}(\mathring {\mb u},0)+ \ti {\mb v}_0(\tau),0,0),\quad
\ti {\mb v}_0(0)= \mathring {\mb v} -{\mb \phi}(\mathring {\mb u},0)\label{tiv0}
\end{align}
and, linearizing, for some $0\leq \theta\leq 1$,
\begin{align*}
\ti{\mb v}_{0,\tau}(\tau) &= \mb g_{,\mb v} (\mr{\mb u}, {\mb \phi}(\mathring {\mb u},0)+ \theta (\ti {\mb v}_0(\tau)-{\mb \phi}(\mathring {\mb u},0)),0,0)\ti{\mb v}_0(\tau), \\
\ti {\mb v}_0(0)&= \mathring {\mb v} -{\mb \phi}(\mathring {\mb u},0).
\end{align*}
 Now, by \eqref{mug}, $s(\mb g_{,\mb v} (\mr{\mb u}, {\mb \phi}(\mathring {\mb u},0),0,0)) \leq -\kappa',$ so that
$\mb g_{,\mb v} (\mr{\mb u}, {\mb \phi}(\mathring {\mb u},0),0,0)$ satisfies the exponential dichotomy, as a constant matrix. Then, using     \eqref{A5} and \cite[Theorem 2.6]{Lin},  we see that there is $\tau_0$ such that $\mb g_{,\mb v} (\mr{\mb u}, {\mb \phi}(\mathring {\mb u},0)+ \theta (\ti {\mb v}_0(\tau)-{\mb \phi}(\mathring {\mb u},0)),0,0)$ satisfies the exponential dichotomy for $\tau\geq \tau_0$ (we can use the constant $\kappa$ from \eqref{tubeest}). Thus
\begin{equation}
\|\ti{\mb v}_0(\tau)\|\leq c_1\|\ti{\mb v}_0(\tau_0)\|e^{-\kappa(\tau-\tau_0)} \leq Ce^{-\kappa\tau}
\label{19}
\end{equation}
for some constants $c_1,C$.
\end{remark}
Finally, we adopt assumption that will ensure that the estimates are valid uniformly on $\mbb R_+$.

\noindent (A5) We assume that the matrix
\begin{equation}
\mc J_{\mb f}(\bar{\mb u}(t)):=\mb f_{,\mb u}(\ups(t),0) + \mb f_{,\mb v}(\ups(t),0){\mb\phi}_{,\mb u}(\bar{\mb u}(t),t), \quad t\in [0,\infty)
\label{expdich2}
\end{equation}
has the exponential dichotomy property \eqref{expdich1} with constants, say, $K_1, \alpha_1$.
\begin{remark} \label{rem3}
Direct verification of \eqref{expdich2} is usually quite difficult. It simplifies, however, if \eqref{CE1} is autonomous. Then $\mb \phi$ is independent of time and thus \eqref{deqeq2} is also autonomous. If $\bar{\mb u}(t) \to \bar{\mb u}^*$ as $t\to \infty$, then $\bar{\mb u}^*$ is an equilibrium of \eqref{deqeq2}. Denote $\bar{\mb\upsilon}^* = (\bar{\mb u}^*, \mb{\phi}(\bar{\mb u}^*))$. If the real parts of all eigenvalues of the Jacobian
$\mc J_{\mb f}(\bar{\mb u}^*) = \mb f_{,\mb u}(\bar{\mb\upsilon}^*,0) + \mb f_{,\mb v}(\bar{\mb\upsilon}^*,0){\mb\phi}_{,\mb u}(\bar{\mb u}^*)$ are negative, then $\mc J_{\mb f}(\bar{\mb u}^*),$ being a constant matrix, has the exponential dichotomy property. Then, arguing as in Remark \ref{rem2},    $\mc J_{\mb f}(\bar{\mb u}(t))$ also has the exponential dichotomy property.
\end{remark}
\section{Error estimates}

As in the classical proof of the Tikhonov theorem, the estimates are split into estimates in the initial layer and in the bulk part. The first part is done as in \cite[Theorem 2.3]{VasBut} or \cite[Theorem 3.3.1]{BanLach}. In the second part we borrow some ideas from \cite{MarCz} but simplify and extend them.
\subsection{Estimates in the initial layer}
Let $(\mb u_\e, \mb v_\e)$ be the solution to \eqref{CE1}.
\begin{lemma} For  any $\rho>0$ there is $\tau_\rho$ and $\e_\rho>0$ such that for any $0<\e<\e_\rho$ and $t_\rho = \e\tau_\rho$ we have
\begin{subequations}\label{est1}
\begin{align}
\|\mb u_\e(t_\rho) - \bar{\mb u}(t_\rho)\|&\leq \rho, \label {uest1}\\
\|\mb v_\e(t_\rho) - \mb \phi(\bar{\mb u}(t_\rho),t_\rho)\|&\leq \rho.\label{vest1}
\end{align}
\end{subequations}
 \label{lemest1}
 \end{lemma}
 \begin{proof}
 Consider the auxiliary function $\hat{\mb v}_0$ defined by \eqref{hav0}.  Then, by \eqref{A5}, for any $\rho>0$ there is $\tau_\rho$ such that
 \begin{equation}
 \|\hat{\mb v}_0(\tau)-\mb\phi(\mb{\mr u}, 0)\|\leq \frac{\rho}{6}
 \label{es1}
 \end{equation}
 for any $\tau\geq \tau_\rho$.  With the change of the variables $\hat{\mb u}_\e(\tau) = \mb u_\e(t),\hat{\mb v}_\e(\tau) = \mb v_\e(t),$ where $t=\e \tau,$  \eqref{CE1} becomes
 \begin{align}
 \hat{\mb u}_{\e,\tau} &=  \e{\mb f}(\hat{\mb u}_\e,\hat{\mb v}_\e,\e\tau,\e), \quad \hat {\mb u}_\e(0) = \mr {\mb u},\nn\\
 \hat{\mb v}_{\e,\tau} &=  {\mb g}(\hat{\mb u}_\e,\hat{\mb v}_\e,\e\tau,\e), \quad \hat {\mb v}_\e(0) = \mr {\mb v},
\label{modsystau}
\end{align}
which is a regular perturbation of
\begin{align}
 \hat{\mb u}_\tau &=  0, \quad \hat {\mb u}(0) = \mr {\mb u},\nn\\
 \hat{\mb v}_\tau &=  {\mb g}(\hat{\mb u},\hat{\mb v},0,0), \quad \hat {\mb v}(0) = \mr {\mb v}.
\label{modsystau1}
\end{align}
Note that the solution of the last equation is $\hat{\mb v}_0(\tau)$. Thus there is $\e_\rho$ such that for any $0<\e<\e_\rho$ and $\tau \in [0,\tau_\rho]$ we have
\begin{equation}
\|\hat{\mb u}_\e(\tau)- \mr {\mb u}\|\leq \frac{\rho}{6}, \quad \|\hat{\mb v}_\e(\tau)- \hat {\mb v}_0(\tau)\|\leq \frac{\rho}{6},
\label{rp1stest}
\end{equation}
or, returning to the original variable,
\begin{subequations}\label{initest1}
\begin{align}
\|{\mb u}_\e(t)- \mr {\mb u}\|&\leq \frac{\rho}{6},\label{initest1a}\\
 \left\|\hat{\mb v}_\e(t)- \hat {\mb v}_0\left(\frac{t}{\e}\right)\right\|&\leq \frac{\rho}{6},\label{initest1b}\end{align}
\end{subequations}
uniformly for $t\in [0,\e\tau_\rho].$ Using  \eqref{initest1a} and the continuity of $\mb \phi$, we may take $\e_0$ small enough to ensure that
\begin{equation}
\|\mb\phi({\mb u}_\e(t),t)-  \mb\phi(\mr{\mb u},0)\|\leq \frac{\rho}{6}
\label{initest2}
\end{equation}
on $[0,\e\tau_\rho]$. Hence, for $t_\rho := \e\tau_\rho,$
\begin{align}
\|\mb v_\e(t_\rho) - \mb\phi({\mb u}_\e(t_\rho),t_\rho)\| &\leq \left\|{\mb v}_\e(t_\rho)- \hat {\mb v}_0\left(\tau_\rho\right)\right\|+\|\hat {\mb v}_0\left(\tau_\rho\right)-\mb\phi(\mb{\mr u},0)\|\nn\\
&\phantom{xxx}+\|\mb\phi(\mb{\mr u},0)-\mb\phi({\mb u}_\e(t_\rho),t_\rho)\|\leq \frac{\rho}{2}.
\label{initest3}
\end{align}
Next,  since $\bar{\mb u}$ clearly is a continuous function, for sufficiently small $\e$ we have  $$\|\bar{\mb u}(t_\rho)- \mr {\mb u}\|\leq \frac{5\rho}{6}$$ and hence, by \eqref{initest1a},
\begin{equation}
\|{\mb u}_\e(t_\rho)- \bar {\mb u}(t_\rho)\|\leq \|{\mb u}_\e(t)- \mr {\mb u}\|+ \|\bar{\mb u}(t)- \mr {\mb u}\| \leq {\rho}.
\label{initest4}
\end{equation}
Using again \eqref{initest1a} and the continuity of $\mb \phi$, for sufficiently small $\e$ we have
\begin{equation}
\|\mb\phi({\mb u}_\e(t),t)-  \mb\phi(\bar{\mb u}(t),t)\|\leq \frac{\rho}{2}
\label{initest5}
\end{equation}
on $[0,\e\tau_\rho]$ and hence, by \eqref{initest3},
\begin{align}
\|\mb v_\e(t_\rho) - \mb\phi(\bar{\mb u}(t_\rho),t_\rho)\| &\leq \|\mb v_\e(t_\rho) - \mb\phi({\mb u}_\e(t_\rho),t_\rho)\| + \|  \mb\phi(\bar{\mb u}(t_\rho),t_\rho)- \mb\phi({\mb u}_\e(t_\rho),t_\rho)\|\nn\\
&\leq \rho.
\label{initest6}
\end{align}
 \end{proof}
 \begin{corollary}
 For any $\rho>0$ there are $\e_\rho$ and $\tau_\rho $, such that
 \begin{equation}
 \left\|\mb v_\e(t) - \mb \phi(\bar{\mb u}(t),t)- \ti{\mb v}_0\left(\frac{t}{\e}\right)\right\|\leq \rho\label{vest3}
 \end{equation}
 for $t \in [0,\e\tau_\rho]$ and $\e <\e_\rho$. \end{corollary}
 \begin{proof} Using \eqref{hatv} and the fact that \eqref{initest1b} and \eqref{initest2} are valid on $[0,t_\rho] =[0, \e\tau_\rho]$, we obtain as in \eqref{initest3}
\begin{align}
\left\|\mb v_\e(t) - \mb\phi({\mb u}_\e(t),t)-\ti{\mb v}_0\left(\frac{t}{\e}\right)\right\| &\leq \left\|{\mb v}_\e(t)- \hat {\mb v}_0\left(\frac{t}{\e}\right)\right\|+\left\|\hat {\mb v}_0\left(\frac{t}{\e}\right)-\mb\phi(\mb{\mr u},0)-\ti{\mb v}_0\left(\frac{t}{\e}\right)\right\|\nn\\
&\phantom{x}+\|\mb\phi(\mb{\mr u},0)-\mb\phi({\mb u}_\e(t),t)\|\leq \frac{\rho}{3},
\label{initest3a}
\end{align}
as the middle term on the right hand side of the inequality is 0 by \eqref{hatv}.
 \end{proof}

\subsection{Large time estimates} By the results of the previous section, we see that we can focus on some neighbourhood of $\bar{\mb \upsilon}(t)$, $t\in [0,\infty)$.
Following some ideas from \cite{MarCz} (based on  \cite{Carr}), we localize the system around $\bar{\mb \upsilon}(t)$, using, however, a different localization that allows for a better control of the linear part of the problem.

Let $x\mapsto \psi(x)$ be a $C^\infty_0(\mbb R_+)$  function equal to 1 for $x \in\left[0,\frac{1}{4}\right)$  and 0 for $x\geq 1$. Then the functions
     $\psi_{\delta}(\mb u,t)= \psi(\delta^{-2}\|\mb u-\bar{\mb u}(t)\|^2)$ and $\chi_{\delta}(\mb v,t) =  \psi(\delta^{-2}\|\mb v-\phi(\bar{\mb u}(t),t)\|^2)$
     are  $C^\infty$ functions (we use the Euclidean norms) in $\mb u$ and $\mb v$, with $\psi_{\delta}(\mb u,t)$ equal to 1 on $E_{\delta/2} (\bar {\mb u}(t), t)$  and 0 outside
    $E_{\delta} (\bar {\mb u}(t), t)$ and  $\chi_{\delta}(\mb u,t)$ equal to 1 on
    $E_{\delta/2} (\phi(\bar{\mb u}(t), t)$ and 0 outside  $E_{\delta} (\phi(\bar{\mb u}(t), t).$ Then
    $$
    \Psi_{\delta}(\mb u,\mb v,t) := \psi_{\delta}(\mb u,t)\chi_{\delta}(\mb v,t)
   $$
     is a $C^\infty$ function of $\mb u$ and $\mb v$ that satisfies $\Psi_{\delta}^t(\mb u,\mb v) =1$ on $E_{\delta/2} (\ups(t))$ and $\Psi_{\delta}^t(\mb u,\mb v) =0$ on $E_{\delta} (\ups( t)).$ By construction,
 \begin{equation}\|\Psi_{\delta,\mb u}\|\leq C\delta^{-1}, \qquad \|\Psi_{\delta, \mb v}\|\leq C\delta^{-1}
 \label{Psiv}
 \end{equation} with a constant $C$ independent of $t$, see also \cite[Section 1.2]{Vlad1}.
 Furthermore,
 \begin{align*}
 \psi_{\delta,t}(\mb u,t) &= -\frac{2}{\delta^2}\phi_{,x}(\delta^{-2}\|\mb u-\bar{\mb u}(t)\|^2)(\mb u-\bar{\mb u}(t))\cdot \bar{\mb u}_{,t}(t)\\& = -\frac{2}{\delta^2}\phi_{,x}(\delta^{-2}\|\mb u-\bar{\mb u}(t)\|^2)(\mb u-\bar{\mb u}(t))\cdot {\mb f}(\ups(t),0),\\
  \chi_{\delta,t}(\mb v,t)&= -\frac{2}{\delta^2}\phi_{,x}(\delta^{-2}\|\mb v-\phi(\bar{\mb u}(t),t)\|^2)(\mb v-\mb\phi(\bar{\mb u}(t),t))\cdot [\mb \phi(\bar{\mb u}(t),t)]_{,t}
  \end{align*}
  and, due to
  \begin{align*}[\mb \phi(\bar{\mb u}(t),t)]_{,t}\ &= \mb\phi_{,\mb u}(\bar{\mb u}(t),t)\bar{\mb u}_{,t}(t) + \mb\phi_{,t}(\bar{\mb u}(t),t) \\&= - \mb g^{-1}_{,\mb v}(\ups(t),0)(\mb g_{,\mb u}(\ups(t),0)\mb f(\ups(t),0) + \mb g_{,t}(\ups(t),0)),
  \end{align*}
we see, by \eqref{mug}, the assumptions on $\ups(t)$ and on the derivatives of $\mb g$, that also
\begin{align*}
 \|\psi_{\delta,t}(\mb u,t) \| &\leq C\delta^{-1}, \quad   \|\chi_{\delta,t}(\mb v,t)\|\leq C\delta^{-1}
  \end{align*}
  independently of $t\in [0,\infty)$.
Then we define
\begin{align}
\breve {\mb g}({\mb u},{\mb v},t,\e) &= \psi_{\delta}(\mb u,t){\mb g}_{,\mb u}(\ups(t),0)({\mb u}-\bar {\mb u}(t))+ {\mb g}_{,\mb v}(\ups(t),0)({\mb v}-{\mb \phi}(\bar {\mb u}(t),t))\nn\\
& \phantom{x}+ \Psi_{\delta}({\mb u}, {\mb v},t)\mc H^\#_{\mb g}({\mb u},{\mb v}, \bar {\mb u}(t),{\mb \phi}(\bar {\mb u}(t),t), t,\e).\label{hatg}
\end{align}
Here  $\mc H^\#_{\mb g} = \mc H_{\mb g}+ \mc J_{\mb g},$ where $\mc H_{\mb g}$ is the second order reminder of the expansion of  ${\mb g}(\mb u,\mb v,t,0)$ with respect to $({\mb u},{\mb v})$ around $\ups(t)$ (see e.g. \cite[Proof of Theorem 4.17]{AUMCS}) and $\mc J_{\mb g}$ the first order reminder of the expansion of ${\mb g}(\mb u,\mb v,t,\e)$ in $\e$ around $\e =0$. In particular, by \textit{op. cit}, $\mc H_{\mb g}$ is of order of $\delta^2,$ while $\mc J_{\mb g}$ is of order of $\e$.

We show that there is $\breve\delta<\delta$ such that
\begin{equation}
\sup\limits_{(\mb u,\mb v)\in \mbb R^{n+m}, t\in [0,\infty)} s(\breve {\mb g}_{,\mb v}(\mb u, \mb v ,t,0))\leq - \breve\kappa<0.
\label{tubeest1`}
\end{equation}
 Indeed, for a given $\delta>0$
\begin{align*}
\breve{\mb g}_{,\mb v}(\mb u,\mb v, t,0)&= {\mb g}_{,\mb v}(\ups(t),0) + \Psi_{\delta,\mb v}({\mb u}, {\mb v},t)\mc H_{\mb g}({\mb u},{\mb v}, \bar {\mb u}(t),{\mb \phi}(\bar {\mb u}(t),t), t)\\
& \phantom{x}+ \Psi_{\delta}({\mb u}, {\mb v},t)[\mc H_{\mb g}({\mb u},{\mb v}, \bar {\mb u}(t),{\mb \phi}(\bar {\mb u}(t),t), t)]_{,\mb v}.
\end{align*}
Then, if $(\mb u, \mb v) \in I_{\delta/2}$, then $\breve{\mb g}(\mb u,\mb v, t,0)= {\mb g} (\mb u,\mb v, t,0)$ and  \eqref{tubeest1`} follows from \eqref{tubeest}, while if $(\mb u, \mb v) \notin I_{\delta}$ we have
$\breve{\mb g}_{,\mb v}(\mb u,\mb v, t,0)= {\mb g}_{,\mb v}(\ups(t),0)$ and \eqref{tubeest1`} is the assumption. Finally, if $(\mb u, \mb v) \in I_{\delta}\setminus I_{\delta/2}$, then, by
 \eqref{Psiv} and the properties of $\mc H_{\mb g}$,
\begin{align*}
&\|\Psi_{\delta,\mb v}({\mb u}, {\mb v},t)\mc H_{\mb g}({\mb u},{\mb v}, \bar {\mb u}(t),{\mb \phi}(\bar {\mb u}(t),t), t)+ \Psi_{\delta}({\mb u}, {\mb v},t)[\mc H_{\mb g}({\mb u},{\mb v}, \bar {\mb u}(t),{\mb \phi}(\bar {\mb u}(t),t), t)]_{,\mb v}\|\\
&\leq  \frac{C}{\delta}M_1\frac{\delta^2}{2} + M_2\delta^2 + M_3\delta,
\end{align*}
where $M_i, i=1,2,3,$ only depend on the suprema of, respectively, third and second order derivatives of $\mb g$ with respect to $\mb u$ and $\mb v$ in $I_\delta$. Hence we see that by taking a sufficiently small $\breve\delta <\delta$ (and the corresponding $\Psi_{\breve \delta}$) we obtain  \eqref{tubeest1`}. Thus $$
\breve{\mb g}({\mb u},{\mb v},t,0) = 0
$$
is solvable for any $t\geq 0$ with $\mb v(t) = \breve{\mb{\phi}}(\mb u,t)$ having a global Lipschitz constant $\breve L$.  Clearly,
\begin{equation}
\mb \phi(\bar{\mb u}(t),t)  = \breve{\mb \phi}(\bar{\mb u}(t),t).\label{brphi}\end{equation}
Also, by continuity, for some $\breve\e>0,$
\begin{equation}
\sup\limits_{(\mb u,\mb v)\in \mbb R^{n+m}, t\in [0,\infty), \e\in I_{\breve\e}} s(\breve {\mb g}_{,\mb v}(\mb u, \mb v ,t,\e))=: - \breve\kappa'<0.
\label{tubeest1`}
\end{equation}
To localize $\mb f$ we re-write it in the form more suited for further calculations, namely
$$
{\mb f}(\mb u,\mb v,t,\e) = \mb f(\bar{\mb u}(t) + \mb \zeta(t), \breve{\mb\phi}(\bar{\mb u}(t)+\mb \zeta(t), t) + \mb\eta(t), t,\e),
$$
where
$$
\mb \zeta(t) = \mb u-\bar{\mb u}(t), \quad \mb \eta(t) = \mb v(t)-\breve{\mb\phi}(\bar{\mb u}(t)+\mb \zeta(t), t).
$$
Then, using \eqref{conv},  we define
\begin{align}
\breve{\mb f}(\mb u,\mb v,t,\e) &= {\mb f}(\ups(t),0) + (\mb f_{,\mb u}(\ups(t),0) + \mb f_{,\mb v}(\ups(t),0){\mb\phi}_{,\mb u}(\bar{\mb u}(t),t))\mb \zeta(t)\nn\\ &\phantom{x}+\mb f_{,\mb v}(\ups(t),0)\mb \eta(t)+  \psi_{\delta}({\mb u},t)\mc H^\#_{\mb f}(\mb \zeta(t),\mb \eta(t),\bar{\mb u}(t),{\mb\phi}(\bar{\mb u}(t),t),t,\e)\nn\\
&={\mb f}(\ups(t),0) + (\mb f_{,\mb u}(\ups(t),0) + \mb f_{,\mb v}(\ups(t),0){\mb\phi}_{,\mb u}(\bar{\mb u}(t),t))\mb \zeta(t)\nn\\
& \phantom{x}+ \mb M_1\mb \eta(t) + \mb M_2(\mb\zeta(t),\mb\eta(t)) + \e\mb M_3(\mb\zeta(t),\mb\eta(t)),     \label{hatf}
\end{align}
where, as before, $\mc H^\#_{\mb f} =\mc H_{\mb f} +\mc J_\e,$ $\mc H_{\mb f}$ is the second order reminder with respect to $\mb \zeta$ and $\mb \eta$ evaluated at $\e=0$, while $\mc J_{\mb f}$ is the first order reminder with respect to $\e$. Thus, $\mb M_i, i=1,2,3,$  depend only on the derivatives of $\mb f,\mb g$ in $E_{\delta} (\ups( t))$ up to second order and are finite irrespectively of $\mb u,\mb v, \e,t$; in particular,  $\mb M_2 = O(\delta^2+\delta\|\mb\eta\| + \|\mb\eta\|^2)$.

 Then we consider the localized problem
\begin{align}
{\mb u}_{\e,t} &= \breve {\mb f}({\mb u}_\e,{\mb v}_\e,t,\e), \quad \breve {\mb u}(t_0) = \breve {\mb u}_0,\nn\\
\e {\mb v}_{\e,t} &=  \breve{\mb g}({\mb u}_\e,{\mb v}_\e,t,\e), \quad \breve {\mb v}(t_0) = \breve {\mb v}_0,
\label{modsys}
\end{align}
for some $t_0\geq 0$. We have the following lemma.
\begin{lemma}
\begin{description}
\item a) Let $(\breve{\mb u}_\e, \breve{\mb v}_\e)$ be the solution to \eqref{modsys}. Then,  $(\breve{\mb u}_\e(t), \breve{\mb v}_\e(t)) = ({\mb u}_\e(t), {\mb v}_\e(t))$ as long as $(\breve{\mb u}_\e(t), \breve{\mb v}_\e(t)) \in E_{\breve\delta/2}(\bar {\mb u}(t), \phi(\bar{\mb u}(t), t))$;
    \item b) Solutions to \eqref{modsys} are bounded uniformly with respect to $\e$ and $t$;
        \end{description}\label{lemma5}
    \end{lemma}
    \begin{proof}
    a) The statement is obvious as all the cut-off functions are equal to one if $(\breve{\mb u}_\e(t), \breve{\mb v}_\e(t)) \in E_{\breve\delta/2}(\bar {\mb u}(t), \phi(\bar{\mb u}(t), t))$ and then \eqref{modsys} coincides with \eqref{CE1}.

    \noindent
    b) The equation for $\breve{\mb v}_\e$ can be written as
    $$
    \breve {\mb v}_{\e,t} = \frac{1}{\e}{\mb g}_{,\mb v}(\ups(t),0)\breve {\mb v}_\e +\frac{1}{\e}\mb h_\e(t) , \quad {\mb v}(0) = \mr {\mb v}
    $$
    where, by the definition of $\breve{\mb g}$,  $\|\mb h_\e(t)\|\leq M$ for some $M$ independent of $t$ and $\e$. Hence, by Lemma \ref{lemeigenest},
    \begin{align}
    \|\breve {\mb v}_{\e}(t)\|&\leq ce^{-\frac{\kappa'}{2\e}(t-t_0)}\| \breve {\mb v}_0\| + \frac{ce^{-\frac{\kappa'}{2\e}(t-t_0)}}{\e}\cl{t_0}{t}e^{\frac{\kappa'}{2\e}s}\|\mb h_\e(s)\|ds\nn\\&\leq ce^{-\frac{\kappa'}{2\e}t} \|\breve {\mb v}_0\| + cM{e^{-\frac{\kappa'}{2\e}t}}\cl{0}{t/\e}e^{\frac{\kappa'}{2}\sigma}d\sigma \nn\\
    &\leq ce^{-\frac{\kappa'}{2\e}(t-t_0)} \|\breve {\mb v}_0\| + \frac{2cM}{\kappa'}(1- e^{-\frac{\kappa'}{2\e}t}) \leq \|\breve {\mb v}_0\| + \frac{2cM}{\kappa'}=:\eta_\infty.\label{eta1}
    \end{align}
    Similarly,  by \eqref{hatf} and the definition of $\mb \zeta,$ we find that
    \begin{align}
    \mb \zeta_{\e,t} &= (\mb f_{,\mb u}(\ups(t),0) + \mb f_{,\mb v}(\ups(t),0){\mb\phi}_{,\mb u}(\bar{\mb u}(t),t))\mb \zeta_\e+ \mb M_1\mb \eta_\e(t) + \mb M_2(\mb\zeta_\e(t),\mb\eta_\e(t)), \nn\\
    &\phantom{x} + \e\mb M_3(\mb\zeta_\e(t),\mb\eta_\e(t)),\nn\\\mb \zeta(t_0)&=\breve{\mb \zeta}_0.
    \label{deltaeq1}
    \end{align}
     Thus, using assumption (A5), for some constants  $C_1,C_2,C_3$
    \begin{equation}
    \|\mb \zeta_\e(t)\| \leq K_1e^{-\alpha_1 ( t-t_0)}\|\breve{\mb \zeta}_0\| + C_1\eta_\infty + C_2(\delta^2+ \delta\eta_\infty +\eta^2_\infty)+C_3\e.
    \label{deltaest1}
    \end{equation}
    In particular, $\mb u_\e(t)$ is bounded uniformly in $t$ and $\e$.
        \end{proof}
\begin{lemma} There is $\rho_0$ such that for any $0<\rho<\rho_0$ there is $\e_\varrho>0$ and constants $c,C,C_4,C_5$ such that for any $\e\in (0,\e_\rho)$ we have
\begin{subequations}\label{ltest}
\begin{align}
\|\mb v_\e(t)- \mb{\phi}(\bar{\mb u}(t),t)\|&\leq c\rho +C\e,\label{ltest1}\\
 \|\mb u_\e(t)-\bar{\mb u}(t)\|&\leq C_4\rho + C_5\e, \label{ltest2}
\end{align}
\end{subequations}
for $t \in [\e\tau_\rho,\infty)$.
\end{lemma}
\begin{proof}
 Using Lemma \ref{lemest1} with arbitrary $\rho<\frac{\breve\delta}{2}$ we consider \eqref{modsys} on $[t_{\rho},\infty)$; that is, with the initial conditions $\breve {\mb u}(t_\rho) =  {\mb u}_{\e}(t_\rho)$ and $\breve {\mb v}(t_\rho) =  {\mb v}_{\e}(t_\rho)$ for arbitrary fixed $\e<\e_\rho$. The initial conditions belong to  $I_\rho \subset I_{\breve\delta/2}$.
 As the first step, we consider a modified approximation for $\breve{\mb v}_\e$ whose error is defined by
\begin{equation}
\breve{\mb \eta}_\e(t) = \breve{\mb v}_\e(t) - \breve{\mb \phi}(\breve{\mb u}_\e(t),t) ,\quad t\geq t_\rho.
\label{ereat1}
\end{equation}
where $\breve{\mb u}_\e$ is the exact solution. We have,
\begin{equation}
\|\breve{\mb \eta}_\e(t_\rho)\| \leq {\rho}.
\label{eta0}
\end{equation}
Then
\begin{align*}
&\e\breve{\mb{\eta}}_{\e,\tau}(t) =\breve{{\mb g}}(\breve{{\mb u}}_\e(t), \breve{\mb {\eta}}_\e + \breve{\mb \phi}(\breve{\mb u}_\e(t),t),t,\e) -
\e [\breve{\mb \phi}(\breve{\mb u}_\e(t),t)]_{,t} \nn\\
&\phantom{xx}=\breve{ {\mb g}}_{,\mb v}(\breve{{\mb u}}_\e(t), \mb v^*(t),t,0)\breve{\mb {\eta}}_\e(t)  -\e [\breve{\mb \phi}(\breve{\mb u}_\e(t),t)]_{,t} + \e\breve{{\mb g}}_\e(\breve{{\mb u}}_\e(t), \breve{\mb {\eta}}_\e + \breve{\mb \phi}(\breve{\mb u}_\e(t),t),t,\e^*),
\end{align*}
where $\mb v^*$ is some point between $\breve{\mb v}_\e$ and the approximation, $\e^*$ is an intermediate point between 0 and $\e$ and we used $$\breve{{\mb g}}(\breve{{\mb u}}_\e(t),  \breve{\mb \phi}(\breve{\mb u}_\e(t),t),t,0)  \equiv 0.$$ Next we observe that
\begin{align*}
[\breve{\mb \phi}(\breve{\mb u}_\e(t),t)]_{,t} &=  \breve{\mb g}^{-1}_{,\mb v}(\breve{\mb u}_\e(t),\breve{\mb \phi}(\breve{\mb u}_\e(t),t),t,0)(\breve{\mb g}_{,\mb u}(\breve{\mb u}_\e(t),\breve{\mb \phi}(\breve{\mb u}_\e(t),t),t,0)
\breve{\mb f}(\breve{\mb u}_\e(t), \breve{\mb v}_\e(t), t,0) \\
&\phantom{xxx} + \breve{\mb g}_{,t}(\breve{\mb u}_\e(t),\breve{\mb \phi}(\breve{\mb u}_\e(t),t),t,0)
\end{align*}
is bounded in $t$ and $\e$ by Lemma \ref{lemma5} (since the solutions are bounded) and \eqref{tubeest1`}. Also $\breve{\mb \phi}(\breve{\mb u}_\e(t),t),t,\e^*)$ is bounded on solutions.  Using again \eqref{tubeest1`} and integrating as in \eqref{eta1}, we obtain that
\begin{equation}
\|\breve{\mb\eta}_\e(t)\| \leq c{\rho} + C\e.
\label{etaest1}
\end{equation}
Before we move on, we observe that by, say \cite[Theorem 2.6]{Lin}, the exponential dichotomy \eqref{expdich2} is satisfied in some $I_{\delta'}$, possibly with different constants $K_2,\alpha_2$. Then, noting that we can decrease the tube $I_{\breve\delta}$ without changing the constants (that can be left as they were for the larger set), we take $\rho, \e_\rho$ and $\breve\delta$ small enough that for $\e <\e_\rho$
\begin{align*}
&K_1\|\breve{\mb \zeta}_0\| + C_1\eta_\infty + C_2(\delta^2+ \delta\eta_\infty +\eta^2_\infty)+C_3\e\\
& \leq K_1\rho + C_1(c{\rho} + C\e) + C_2(\breve\delta^2+\breve\delta + (c{\rho} + C\e)^2) +C_3\e <\delta';
\end{align*}
that is, $\breve{\mb \zeta_\e}$ stays in the region where $\breve {\mb f}_{,\mb u}+ \breve{\mb f}_{,\mb v}\breve{\mb \phi}$  has the exponential dichotomy property. Then, for
$$
\breve{\mb\zeta}_\e(t) = \breve{\mb u}_\e(t)-\bar{\mb u}(t),
$$
we have, by $\mb f(\bar{\mb u}(t),
\mb{\phi}(\bar{\mb{ u}}(t),t),t,0) = \breve{\mb f}(\bar{\mb u}(t),
\mb{\phi}(\bar{\mb{ u}}(t),t),t,0)$,
\begin{align*}
&\breve{\mb \zeta}_{\e,t}(t) = \breve{\mb u}_{\e,t}(t) - \bar{\mb u}_{t}(t) \\
& = \breve{\mb f}(\breve{\mb u}_\e(t),\breve{\mb v}_\e(t),t,0) - \breve{\mb f}(\bar{\mb u}(t) +
\mb{\phi}(\bar{\mb{ u}}(t),t),t,0)+\e\breve{\mb f}_\e(\breve{\mb u}_\e(t),\breve{\mb v}_\e(t),t,\e^*)\\
&= \breve{\mb f}(\bar{\mb u}(t)+\breve{\mb\zeta}_\e(t),\breve{\mb\phi}(\bar{\mb u}(t) + \breve{\mb\zeta}_\e(t),t,0)+\breve{\mb \eta}_\e(t),t,0) - \breve{\mb f}(\bar{\mb u}(t), \mb\phi(\mb{\bar u}(t),t),t)\\
&\phantom{xx}+\e\breve{\mb f}_\e(\breve{\mb u}_\e(t),\breve{\mb v}_\e(t),t,\e^*)\\
&=(\breve{\mb f}_{,\mb u}( {\mb u}^*(t),\mb {v}^*(t),t,0) + \breve{\mb f}_{,\mb v}( {\mb u}^*(t),\mb {v}^*(t),t,0)\breve{\mb{\phi}}(\mb u^*(t),t))\breve{\mb \zeta}_\e(t) \\
&\phantom{xx}+ \breve{\mb f}_{,\mb v}( {\mb u}^*(t),\mb {v}^*(t),t,0)\breve{\mb \eta}_\e(t)+\e\breve{\mb f}_\e(\breve{\mb u}_\e(t),\breve{\mb v}_\e(t),t,\e^*),
\end{align*}
where again $\mb u^*, \mb v^*$ and $\e^*$ are some intermediate values. Taking into account \eqref{uest1} and \eqref{etaest1} and the fact that $\breve {\mb f}_{,\mb u}+ \breve{\mb f}_{,\mb v}\breve{\mb \phi}$  satisfies the exponential dichotomy property on $({\mb u}^*(t),\mb {v}^*(t),t,0)$, we obtain
\begin{equation}
\|\mb{\zeta}_\e(t)\|\leq C_4\rho + C_5\e, \quad t \in [\e\tau_\rho, \infty), \e \in [0,\e_\rho],
\label{zetaest2}
\end{equation}
for some constants $C_4, C_5$. Then, using \eqref{brphi} and the Lipschitz continuity of $\breve{\mb \phi}$,
\begin{align*}
\|\breve{\mb v}_\e(t)- \mb{\phi}(\bar{\mb u}(t),t)\| &\leq \|\breve{\mb \eta}_\e(t)\| + \|\breve{\mb{\phi}}(\bar{\mb u}(t),t) -  \breve{\mb{\phi}}(\breve{\mb u}_\e(t),t)\|\\
&\leq c\rho + C\e + \breve L(C_4\rho + C_5\e) = C_6 \rho +C_7 \e.
\end{align*}
Finally, selecting $\rho_0$ and the corresponding $\e_{\rho_0}$ so that $C_6\rho+C_7\e <\breve\delta/2$ and  $C_4\rho+C_5\e <\breve\delta/2$, w obtain $\breve{\mb u}_\e(t) = \mb u_\e(t)$ and $\breve{\mb v}_\e(t) = \mb v_\e(t)$ and hence \eqref{ltest} follows.
\end{proof}
\begin{theorem}
There is $\varrho_0>0$ such that for any $0<\varrho<\varrho_0$ there is $\e_\varrho$ such that for any $\e\in [0,\e_\varrho]$ and $t\in [0,\infty)$
\begin{subequations}\label{mest1}
\begin{align}
 \left\|\mb v_\e(t) - \mb \phi(\bar{\mb u}(t),t)- \ti{\mb v}_0\left(\frac{t}{\e}\right)\right\|&\leq \varrho,\label{mest1a}\\
 \left\|\mb u_\e(t) - \bar{\mb u}(t)\right\|&\leq \varrho.\label{mest1b}
  \end{align}
\end{subequations}
\label{mthm1}
\end{theorem}\begin{proof} Let us take arbitrary  $\rho>0$ and let $\tau_\rho$ and $\e_\rho$ be the corresponding values of $\tau$ and $\e$ determined in Lemma \ref{lemest1}. Let us take any $\e<\e_\rho$. Then noting that, by \eqref{initest1a}, inquality \eqref{initest4} holds uniformly on $[0,\e\tau_\rho]$, we see that \eqref{mest1b} holds   on $[0,\e\tau_\rho]$ as long as $\rho<\varrho$. Then \eqref{zetaest2} implies  \eqref{mest1b} on $[\e\tau_\rho, \infty)$ for $\e\in [0,\e_\varrho],$ provided $\rho$ and $\e_\rho$ are such that   $C_4\rho + C_5\e_\rho\leq \varrho \leq \breve\delta/2$.

Further, by \eqref{es1}, for any $\rho'>0$ and appropriate $\tau_{\rho'}$ and $\e_{\rho'}$ we have $$
\left\|\ti{\mb v}_0\left(\frac{t}{\e}\right)\right\|\leq \frac{\rho'}{6}
$$
for $t\geq \e\tau_{\rho'}, \e<\e_{\rho'},$ and thus, by \eqref{etaest1}, for this range of $t$ and $\e$ we have
$$
\left\|\mb v_\e(t) - \mb \phi(\bar{\mb u}(t),t)- \ti{\mb v}_0\left(\frac{t}{\e}\right)\right\|\leq c{\rho'} + C\e + \frac{\rho'}{6}.
$$
Combining the above with \eqref{vest3} and selecting $\rho'\leq \varrho$ and $\e_{\rho'}$ such that $c\frac{\rho'}{2} + C\e\leq \breve\delta/2$ and  $c\frac{\rho'}{2} + C\e + \frac{\rho'}{6}\leq \varrho$ for $\e \in [0, \e_{\rho'}]$ we obtain \eqref{mest1a}. Thus, \eqref{mest1} holds for $\e\in [0, \min\{\e_\rho,\e_{\rho'}\}].$
\end{proof}

\section{An application to derive an Allee type dynamics}

A population displays the so-called Allee type dynamics if it has some carrying capacity to which it monotonically increases if it is large enough but goes extinct  if it is too small. Mathematically, the equation describing the evolution of the population should have three equilibria: asymptotically stable 0 as the extinction equilibrium, the repelling threshold equilibrium and the attractive carrying capacity. One of the ways to derive equations of this type is to look at populations interacting with each other according to the mass action law and exploiting multiple time scales occurring in such models. We present an example introduced in \cite{Thieme} and further analysed in \cite{BanLach}. In this model we consider a population $N$ of females subdivided into subpopulations $N_1$ of females who recently have mated and $N_2$  of females who are searching for a
mate. We assume that there is an equal number $N$ of males. If the females reproduce in a
very short time after mating, then
\begin{align}
{N}_{1,t}  &=\beta N_1-(\mu+\nu N)N_1-\sigma N_1+\xi NN_2,\nn\\
{N}_{2, t}  &=-(\mu+\lambda+\nu N)N_2+\sigma N_1-\xi NN_2, \label{eq2}
\end{align}
where $\beta$ is the reproduction rate of the recently
mated females, $\mu+\nu N$ is the mortality rate
of the recently mated females, $\mu+\lambda+\nu N$ denotes the (increased) mortality rate of the searching females, $\sigma$
denotes the rate at which the females switch from the reproductive
stage to the searching stage and $\xi N$ denotes the per capita
rate at which a searching female finds one out of $N$ potential
mates.

To nondimensionalize the system, we rescale the time in units of the natural life  expectancy $1/\mu,$ $s = \mu t,$  and, assuming $\beta-\mu>0,$  we introduce the carrying capacity
$$
K=\frac{\beta-\mu}{\nu}
$$
and setting $N_1=xK,$  $N_2=yK$ and $z = x+y$, we obtain our system in
dimensionless form,
\begin{equation}
\begin{aligned}
\mu x_{,s} & =(\beta-\mu)x(1-z)-\sigma x +\xi Kyz,\\
\mu {y}_{,s} & =-(\mu+\lambda +\nu Kz)y+\sigma x -\xi Kyz,
\end{aligned}
\label{huha}
\end{equation}
where $s$ is the rescaled time. Let us denote $\varepsilon=\frac{\mu}{\sigma};$
 that is, $\varepsilon$ is the ratio of the average time of satiation to the average life span. In
many cases it is a small parameter (for instance, for wild rabbits the average lifespan is 4 years and they breed 6-7 times per year, giving $\e \approx 0.04$). For the population not to become extinct, we can argue that the rate at which a
female finds a mate should be comparable with the rate she switches to a searching mode
after reproduction, see \cite{BanLach} for a discussion of other cases. Thus, writing $\xi/\mu  = \bar{\xi}/\e$ and denoting
$R_0=\frac{\beta}{\mu},$
 we consider
\begin{eqnarray}
z_{,s} & = & (z-y)(R_0-1)(1-z)-\left(1+\frac{\lambda +\nu Kz}{\mu}\right)y,\quad z(0)= \mr{z}\nonumber\\
y_{,s} & = & -\left(1+\frac{\lambda +\nu Kz}{\mu}\right)y -\frac{\bar \xi K}{\e}yz +\frac{1}{\e} (z-y), \quad y(0) =\mr{y}.
 \label{allsys3}
\end{eqnarray}
The right hand side of the first equation can be simplified as
\begin{eqnarray*}
&&(z-y)(R_0-1)(1-z)-\left(1+\frac{\lambda +\nu
Kz}{\mu}\right)y= (R_0-1)z(1-z) - \frac{\beta+\la}{\mu}y
\end{eqnarray*}
so that we finally consider
\begin{eqnarray}
z_{,s} &=&(R_0-1)z(1-z) - \frac{\beta+\la}{\mu}y, \quad z(0)= \mr {z}, \nn\\
\varepsilon y_{,s}&=&-\e\left(1+\frac{\lambda +\nu Kz}{\mu}\right)y -{\bar \xi K}yz +z-y,\quad y(0) = \,\mr {y}.
 \label{allsys4}
\end{eqnarray}
We obtain the quasi steady state as
\begin{equation}
y = \phi(z) = \frac{z}{1 +\bar \xi K z} \label{qss1}
\end{equation}
and hence the reduced equation is given by
\begin{equation}
{\bar{z}}_{,s} = (R_0-1)\bar z (1-\bar z) -
\frac{\beta+\la}{\mu}\frac{\bar z}{1 +\bar\xi K \bar z}, \quad z(0) = \mr z.
\label{allee1}
\end{equation}
The auxiliary equation \eqref{hav0} here takes the form
\begin{equation}
\frac{\mathrm{d}\,\hat y}{\mathrm d\tau} = -\hat y(1+{\bar \xi K}\mr z) +\mr z.
\label{allaux}
\end{equation}
We are interested in $\mr z \in [0,\infty)$. Then the only equilibrium of
(\ref{allaux}) is  $\hat y = \phi(z)\geq 0$ and the right
hand side is decreasing for $\mr z> - 1/\bar \xi K$. So, \eqref{mug} is satisfied.
Further, by the above, any
nonnegative $\mr {y}$ belongs to the domain of attraction of the
equilibrium solution.

Let us have a closer look at (\ref{allee1}) and find out what dynamics it describes. The stationary points are determined from the equation
$$
0=z\left((R_0-1)(1-z) -
\frac{\beta+\la}{\mu}\frac{1}{1 +\bar\xi Kz}\right).
$$
This immediately gives $z_1=0.$ Then the zeroes of the expression in the brackets can be determined as the solutions to
\begin{equation}
g(z):= (1-z)(1+\bar\xi Kz) = \frac{\beta+\la}{\beta-\mu}.
\label{All}
\end{equation}
The graph of $g$ is a downward parabola with the roots at  $z=1$ and $z = -1/\bar\xi K$ and the vertical axis intercept at $1$. It takes the maximum of $\frac{1}{4}\left(1+ \frac{1}{\bar\xi K}\right)^2$ at $z_{max} = \frac{\bar\xi K-1}{2\bar\xi K}$. The model exhibits the Allee effect if and only if \eqref{All} has two positive solutions $0<z_2<z_3$ (the required stability is ensured by the parabola being downward); that is, if and only if
\begin{equation}
\bar\xi K >1, \quad 1 < \frac{\beta+\la}{\beta-\mu} < \frac{1}{4}\left(1+ \frac{1}{\bar\xi K}\right)^2.
\label{All2}
\end{equation}
In particular, we must have $R_0>1$.

Thus, under \eqref{All2}, any solution to the limit equation \eqref{allee1} with $0\leq \mr z <z_2$ converges to 0, while if $z_2<\mr z\leq z_2$, then the corresponding solution converges to $z_3$. Since the stationary points $0$ and $z_3$ are hyperbolic, Assumption (A5) is satisfied by Remark \ref{rem3} and  we can claim that the solution
$(z_\e(t), y_\e(t))$ of (\ref{allsys4}) with $\mr {z} \in (0,z_2)$ or $\mr {z} \in (z_2,z_3)$ and $\mr {y}>0$
satisfies
\begin{eqnarray*}
\lim\limits_{\e\to 0} z_\e(t) &=& \bar z(t),\quad{\rm on}\;[0,\infty),\\
\lim\limits_{\e\to 0} y_\e(t) &=& \frac{\bar z(t)}{1 +\bar \xi K
\bar z(t)},\quad{\rm on}\;(0,\infty),
\end{eqnarray*}
 where $\bar z$ solves (\ref{allee1}). Furthermore,  the initial layer correction is given by
$$
\hat y\left(\frac{t}{\e}\right) =\, \mr {y} +\frac{\mr {z}}{1+\bar \xi K
\mr {z}}\left(1-e^{-\frac{(1+\bar \xi K\mr z)t}{\e}}\right)
$$
and thus
\begin{eqnarray*}
&&\lim\limits_{\e\to 0} \left(\!y_\e(t)-\frac{\bar z(t)}{1 +\bar \xi K
\bar z(t)}-\left(\mr  {y} +\frac{\mr  {z}}{1+\bar \xi K
\mr  {z}}\left(1-e^{-\frac{(1+\bar \xi K\mr z)t}{\e}}\right)\right) -
\frac{ \mr  {z}}{1 +\bar \xi K \mr  {z}}\! \right) \\
&&\phantom{xxx}= 0
\end{eqnarray*}
uniformly $[0,\infty]$. Thus the long term dynamics of \eqref{allsys3} and \eqref{qss1}, \eqref{allee1} are equivalent.

\begin{figure}
 \centering
 \includegraphics[width=4.5cm,height=6cm,angle=270,bb=50 50 554 770,clip=true,trim=2pt 0pt 0pt 0pt]{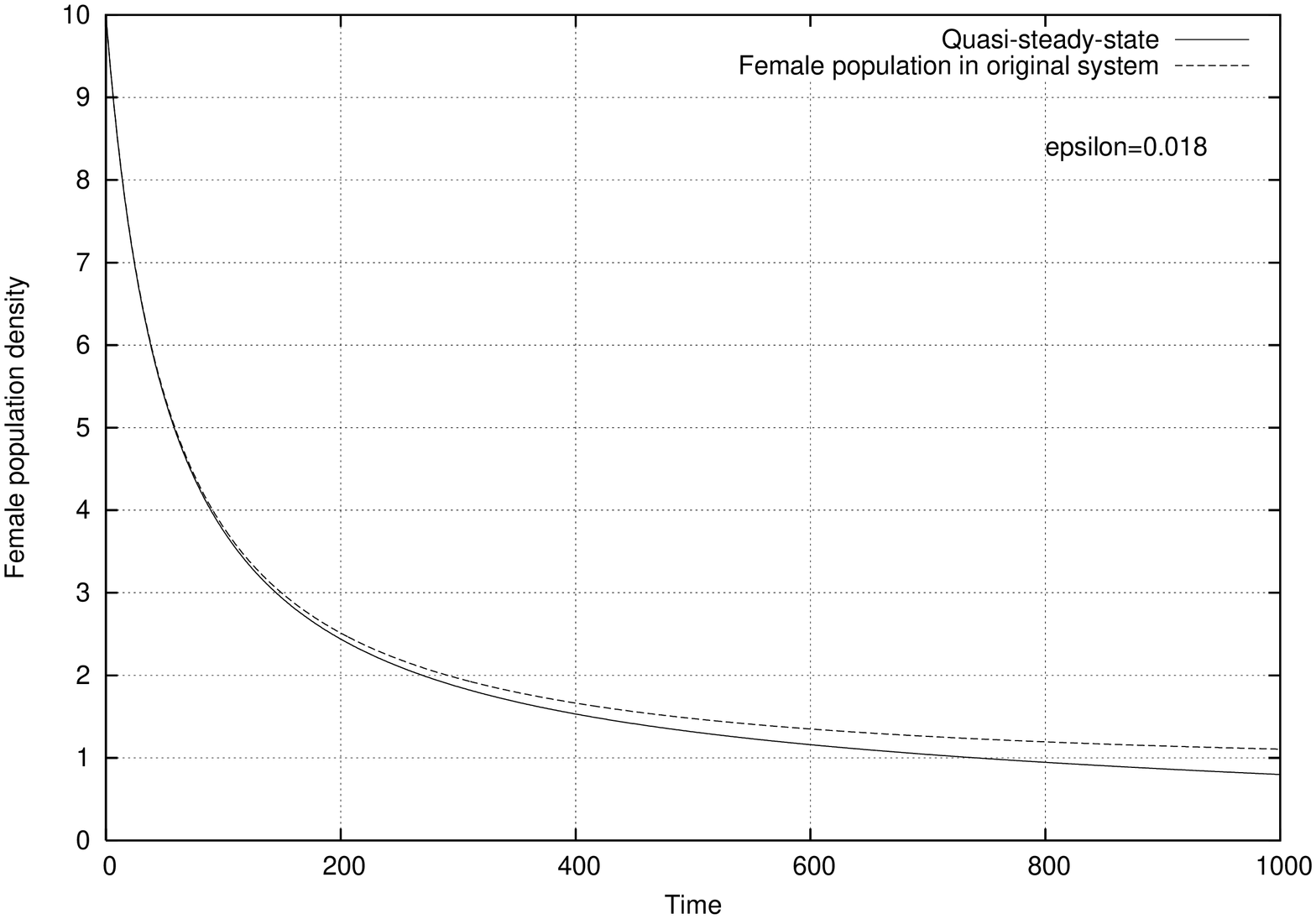} \\
 \includegraphics[width=4.5cm,height=6cm,angle=270,bb=50 50 554 770,clip=true,trim=2pt 0pt 0pt 0pt]{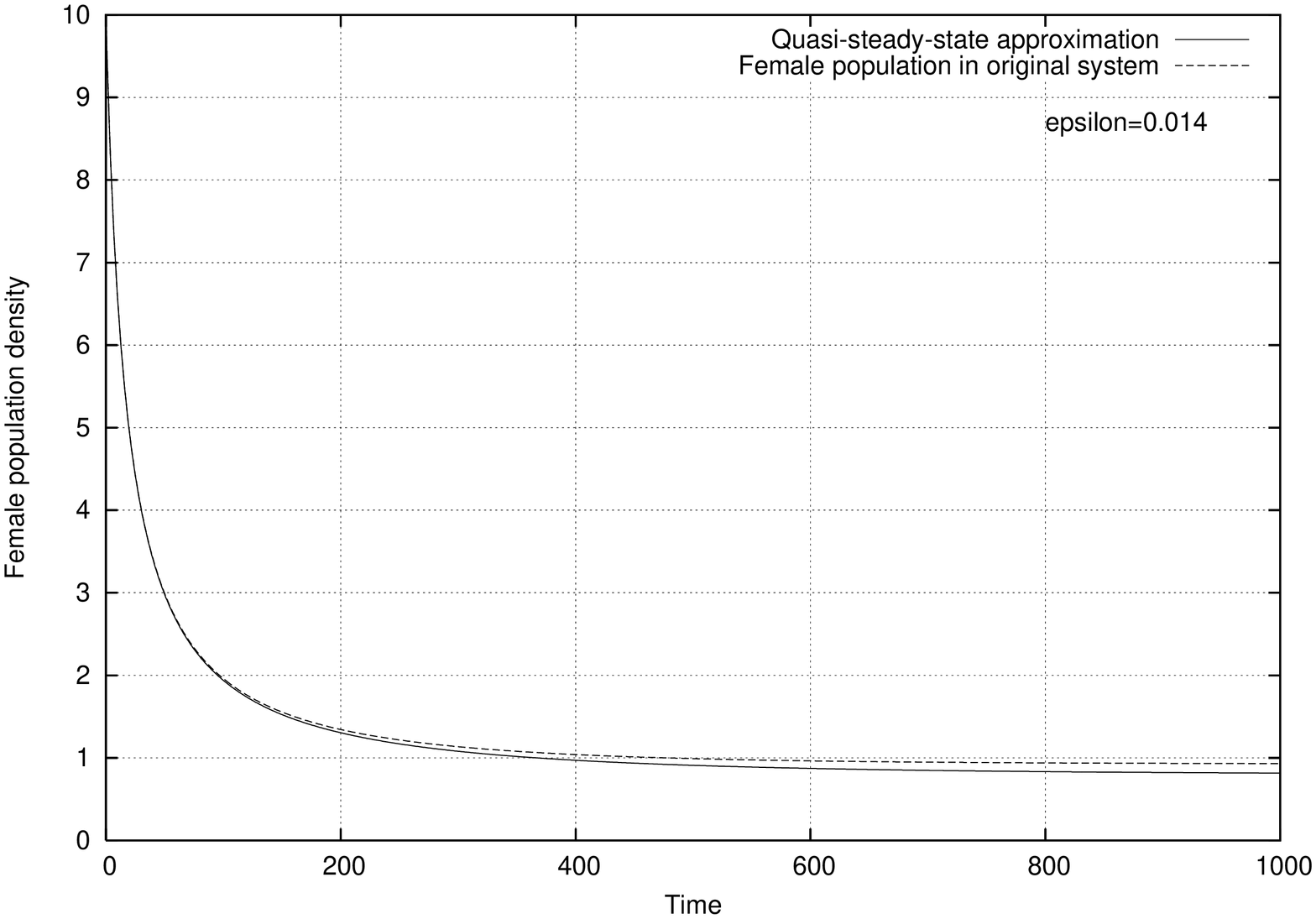}
 % AlleeEffect2.eps: 0x0 pixel, 300dpi, 0.00x0.00 cm, bb=50 50 554 770
 \caption{Comparison of the total female population $z,$ given by (\ref{allsys4}) with the quasi steady state approximation $\bar z$ given by
  (\ref{allee1}) for $\e =0.018$ (top) and $\e=0.014$ (bottom).}
\label{fig:2}
\end{figure}

\appendix
\section{Spectral bound and exponential dichotomy}
In general, the spectral bound of a time dependent matrix $D(t)$ does not determine whether it has the exponential dichotomy property. The situation changes fortunately for singularly perturbed linear equations. We recall the relevant result from \cite{VasBut} together with the proof that originally is in  Russian. Going through the proof is also useful to ascertain that, under some additional assumptions, the result remains valid on $\mbb R_+$.
\begin{lemma} \cite[Lemma 3.2]{VasBut} Assume that $\mbb R_+\ni t\mapsto D(t)$ is a bounded, uniformly continuous $m\times m$ matrix function with eigenvalues $\la_i (t)$, $1\leq i\leq n(t)$, where $n(t)$ is the number of distinct eigenvalues of $D(t)$, satisfying
\begin{equation}
\sup\limits_{t\in \mbb R_+}\Re \la_i(t) < -2\sigma<0.
\label{eigenest}
\end{equation}
Then there is $c>0$ and $\e_0>0$ such that for any $0<\e\leq \e_0$ the fundamental matrix $Y(t,\e)$ of the system
\begin{equation}
\e Y_{t} = D(t)Y, \qquad Y(0,\e) = I_{R^m},\label{LS2}
\end{equation}
satisfies
\begin{equation}
\|Y(t,\e)Y^{-1}(s,\e)\|\leq ce^{-\frac{\sigma(t-s)}{\e}}, \quad 0\leq s\leq t < \infty.
\label{Yest}
\end{equation}
\label{lemeigenest}
\end{lemma}
\begin{proof}
Let $\mc Y(t,s,\e) = Y(t,\e)Y^{-1}(s,\e).$ For a given fixed $t_0\in \mbb R_+$ we write, by \eqref{LS2}, \begin{equation}
\e \mc Y_{t} = D(t_0)\mc Y + (D(t)-D(t_0))\mc Y, \qquad \mc Y(s,s,\e) = I_{R^m}\label{LS2a}
\end{equation}
and, by the variation of constants formula,
\begin{equation}
\mc Y(t,s,\e) = e^{\frac{D(t_0)(t-s)}{\e}} + \cl{s}{t}\frac{1}{\e} e^{\frac{D(t_0)(t-q)}{\e}}(D(q)-D(t_0))\mc Y(q,s,\e)dq.
\label{LS3}
\end{equation}
Since \eqref{LS3} is valid for any $t_0$, it is valid, in particular, for $t_0=t$; that is,
\begin{equation}
\mc Y(t,s,\e) = e^{\frac{D(t)(t-s)}{\e}} + \cl{s}{t}\frac{1}{\e} e^{\frac{D(t)(t-q)}{\e}}(D(q)-D(t))\mc Y(q,s,\e)dq.
\label{LS3}
\end{equation}
Let us define
\begin{equation}
w(t,s,\e) := \|\mc Y(t,s,\e)\|e^{\frac{\sigma(t-s)}{\e}},\quad 0\leq s\leq t<\infty,
\label{wdef}
\end{equation}
where $\sigma$ was defined in \eqref{eigenest}. Then, thanks to the fact that the inequality in \eqref{eigenest} is sharp, $\|e^{D(t)(t-s)}\| \leq c_0e^{-2\sigma( t-s)}$ for some constant $c_0$ (here $D(t)$ is treated as a constant matrix for a fixed $t$). Hence, multiplying \eqref{LS3} by $e^{\frac{\sigma(t-s)}{\e}}$ and taking norms, we obtain
\begin{equation}
w(t,s,\e) \leq c_0  +\frac{c_0}{\e}  \cl{s}{t} e^{-\frac{\sigma(t-q)}{\e}}\|D(q)-D(t)\|w(q,s,\e)dq.
\label{LS4}
\end{equation}
Now, for any fixed $T<\infty$ we define $M(T,\e) := \max\limits_{0\leq s\leq t\leq T}w(t,s,\e)$ so that
\begin{equation}
M(T,\e)\leq c_0 + \frac{c_0M(T,\e)}{\e} \max\limits_{0\leq s\leq t\leq T} J(t,s,\e),
\label{MT}
\end{equation}
where $$
J(t,s,\e) := \cl{s}{t} e^{-\frac{\sigma(t-q)}{\e}}\|D(q)-D(t)\|dq
$$
Next, thanks to the uniform continuity of $D(t)$,
$$
\delta(\e) := \max\limits_{0\leq t-\sqrt{\e} \leq q\leq t<\infty}\|D(q)-D(t)\|
$$
satisfies $\delta(\e) \to 0$ as $\e\to 0$, uniformly in $t$. In other words, if $t-s\leq \sqrt{\e}$, then $\|D(q)-D(t)\|\leq \delta(\e)$ whenever $q \in [s,t],$ irrespective of $s$ and $t$.  Thus, if $t-s\leq \sqrt{\e}$, then
\begin{equation}
J(t,s,\e) \leq \delta(\e) \cl{s}{t} e^{-\frac{\sigma(t-q)}{\e}}dq \leq \frac{\delta(\e)\e}{\sigma}.
\label{Jest}
\end{equation}
On the other hand, if $t-s> \sqrt{\e}$, then we split $J= J_1+J_2,$ where, denoting $c_1 = 2\sup\limits_{t\in \mbb R_+}\|D(t)\|$,
\begin{align*}
J_1(t,s,\e) &= \cl{s}{t-\sqrt{\e}} e^{-\frac{\sigma(t-q)}{\e}}\|D(q)-D(t)\|dq\leq c_1  \cl{s}{t-\sqrt{\e}} e^{-\frac{\sigma(t-q)}{\e}}dq\leq \frac{c_1\e}{\sigma}e^{-\frac{\sigma}{\e}}
\end{align*}
and  $J_2$ satisfies estimate \eqref{Jest}. Thus
$$
\max\limits_{0\leq s\leq t\leq T} J(t,s,\e)\leq \e\left(\frac{\delta(\e)}{\sigma} + \frac{c_1}{\sigma}e^{-\frac{\sigma}{\e}}\right)
$$
and the expression on the right hand side is independent of $T$. We can take $\e_0$ small enough for $\frac{\delta(\e)}{\sigma} + \frac{c_1}{\sigma}e^{-\frac{\sigma}{\e}}\leq \frac{1}{2}$ if $0<\e\leq \e_0$. Thus
$$
M(T,\e) \leq c_0+ \frac{M(T,\e)}{2};
$$
that is, $M(T,\e)\leq 2c_0$ irrespectively of $T$ and $0<\e\leq \e_0$. Thus, by \eqref{wdef},
$$
\|\mc Y(t,s,\e)\|\leq 2c_0e^{-\frac{\sigma(t-s)}{\e}},\quad 0\leq s\leq t<\infty, 0<\e\leq \e_0$$
and the proof is complete.
\end{proof}

%\bibliographystyle{abbrv}
%\bibliography{Dropbox/FCBook/BLLbook/BLL_Bk}
%\bibliography{refsingpert}
\end{document}